%% file: Comb.tex
\documentclass[11pt,oneside]{article}  
\include{maquereaux}

\usepackage{tikz-cd}

\definecolor{couleur}{rgb}{0.4,0.8,0.7}

\begin{document}

\title{The comb representation of compact ultrametric spaces}
\author{\textsc{Amaury Lambert$^{1,2,\natural}$, Ger\'onimo Uribe Bravo$^{3,\dagger}$}
}
\date{\today}
\maketitle

\noindent\textsc{$^1$
UPMC Univ Paris 06\\
Laboratoire de Probabilités et Modèles Aléatoires CNRS UMR 7599\\
Paris, France}\\
\noindent\textsc{$^2$
Collège de France\\
Center for Interdisciplinary Research in Biology CNRS UMR 7241\\
Paris, France}\\
\noindent\textsc{$^3$
UNAM\\
Instituto de Matem\'aticas\\
M\'exico DF, Mexico}\\

\noindent $\natural$
\textsc{E-mail: }amaury.lambert@upmc.fr\\
\textsc{URL: }http://www.lpma-paris.fr/pageperso/amaury/index.htm\\
\noindent $\dagger$
\textsc{E-mail: }geronimo@matem.unam.mx\\
\textsc{URL: }http://www.matem.unam.mx/geronimo/
\\

\begin{abstract}
\noindent
We call a \emph{comb} a map $f:I\to [0,\infty)$, where $I$ is a compact interval, such that $\{f\ge \varepsilon\}$ is finite for any $\varepsilon$. A comb induces a (pseudo)-distance $\dtf$ on $\{f=0\}$ defined by $\dtf(s,t) = \max_{(s\wedge t, s\vee t)} f$. We describe the completion $\bar I$ of $\{f=0\}$ for this metric, which is a compact ultrametric space called \emph{comb metric space}. 

Conversely, we prove that any compact, ultrametric space $(U,d)$ without isolated points is isometric to a comb metric space. 
We show various examples of the comb representation of well-known ultrametric spaces: the Kingman coalescent, infinite sequences of a finite alphabet, the $p$-adic field and spheres of locally compact real trees.
In particular, for a rooted, locally compact real tree defined from its contour process $h$, the comb isometric to the sphere of radius $T$ centered at the root can be extracted from $h$ as the depths of its excursions away from $T$.

\end{abstract}  	
\bigskip

\noindent
\textit{Running head.} Comb metric spaces.

\medskip

\medskip

\section{Introduction}

An ultrametric space is a metric space $(U,d)$ such that for any $x,y,z\in U$, $d(x,z)\le \max (d(x,y), d(y,z))$. The $p$-adic field is a famous example of ultrametric space. Another example is the set of infinite sequences of elements of a finite alphabet, say $\{0,1,\ldots p-1\}$. This set can also be seen as the boundary of the planar, regular $p$-ary rooted tree. More generally for any rooted tree, a sphere centered at the root, i.e. the set of all points at the same (graph) distance to the root, is ultrametric. Characterizing the metric of spheres of trees is of particular interest in population biology (population genetics, phylogenetics), where the tree is a time-calibrated genetic or phylogenetic tree, the sphere is the set of individuals or species living at the same time and the metric structure of such a sphere is the set of ancestral relationships between these individuals or species. The tree spanned by a sphere (and the root) is sometimes called reduced tree or coalescent tree (and is said ultrametric by abuse).\\

We start with the following observation, proved in Section \ref{sec:finite-ultra} of the appendix.
For any finite subset $S$ of an ultrametric space $(U,d)$, there is at least one labelling of its elements $\{x_i:1\le i\le n\}$ such that  for any $i<j$, 
\begin{equation}
\label{eqn:intro}
d(x_i,x_{j}) = \max\{d(x_k,x_{k+1});i\le k<j\}. 
\end{equation}
With this labelling, the metric structure of $S$ is then completely characterized by the list of $n-1$ distances between pairs of consecutive points. We call this list of distances a \emph{comb}. The goal of the present paper is to extend this representation to all compact ultrametric spaces. 

We will further show that for a rooted, locally compact real tree defined from its contour process $h$, the comb isometric to the sphere of radius $T$ centered at the root is the list of depths of excursions of $h$ away from $T$.\\

In the next section, we define a comb and the ultrametric associated with it. We recall how the Kingman coalescent tree can be embedded in a random comb. We then identify the completion of the comb for this metric, that we call \emph{comb metric space}. 
In Section \ref{sec:converse}, we show that any compact ultrametric space without isolated point is isometric to a comb metric space. We provide such a representation for the two classical examples of ultrametric spaces mentioned in the beginning of this introduction, namely the $p$-adic field and the boundary of the infinite $p$-ary tree. 
Last, Section \ref{sec:spheres} is devoted to the special case of spheres of real trees.


\section{The comb metric}

\subsection{Definition and examples}

Let $I$ be a compact interval and $f:I\to \R_+$ such that for any $\varepsilon >0$, $\{f\ge \varepsilon\}$ is finite. 
For any $s,t\in I$, define $\dtf$ by
$$
\dtf(s,t) = \sup_{(s\wedge t, s\vee t)} f.
$$
It is clear that $\dtf$ is a pseudo-distance on $\{f=0\}$, and more precisely, that it is ultrametric, in the sense that
$$\dtf(r,t)\le \max(\dtf(r,s),\dtf(s,t)) \qquad r,s,t\in I.
$$
It is a distance on $\{f=0\}$ whenever $\{f\not=0\}$ is dense in $I$ for the usual distance. This may not be the case in general, so we need to consider $\dot{I}$ the quotient space $\{f=0\}|_\sim$, where $\sim$ is the equivalence relation 
 $$
 s\sim t \Leftrightarrow \dtf(s,t)=0 \Leftrightarrow f=0 \mbox{ on }[s\wedge t,s\vee t].
 $$

\begin{dfn}
\label{dfn:comb}
We call $f$ a \emph{comb-like function} or \emph{comb}, and $\dtf$ the \emph{comb metric} on $\dot{I}$.
\end{dfn}

\begin{figure}[!ht]
\includegraphics[width=\textwidth]{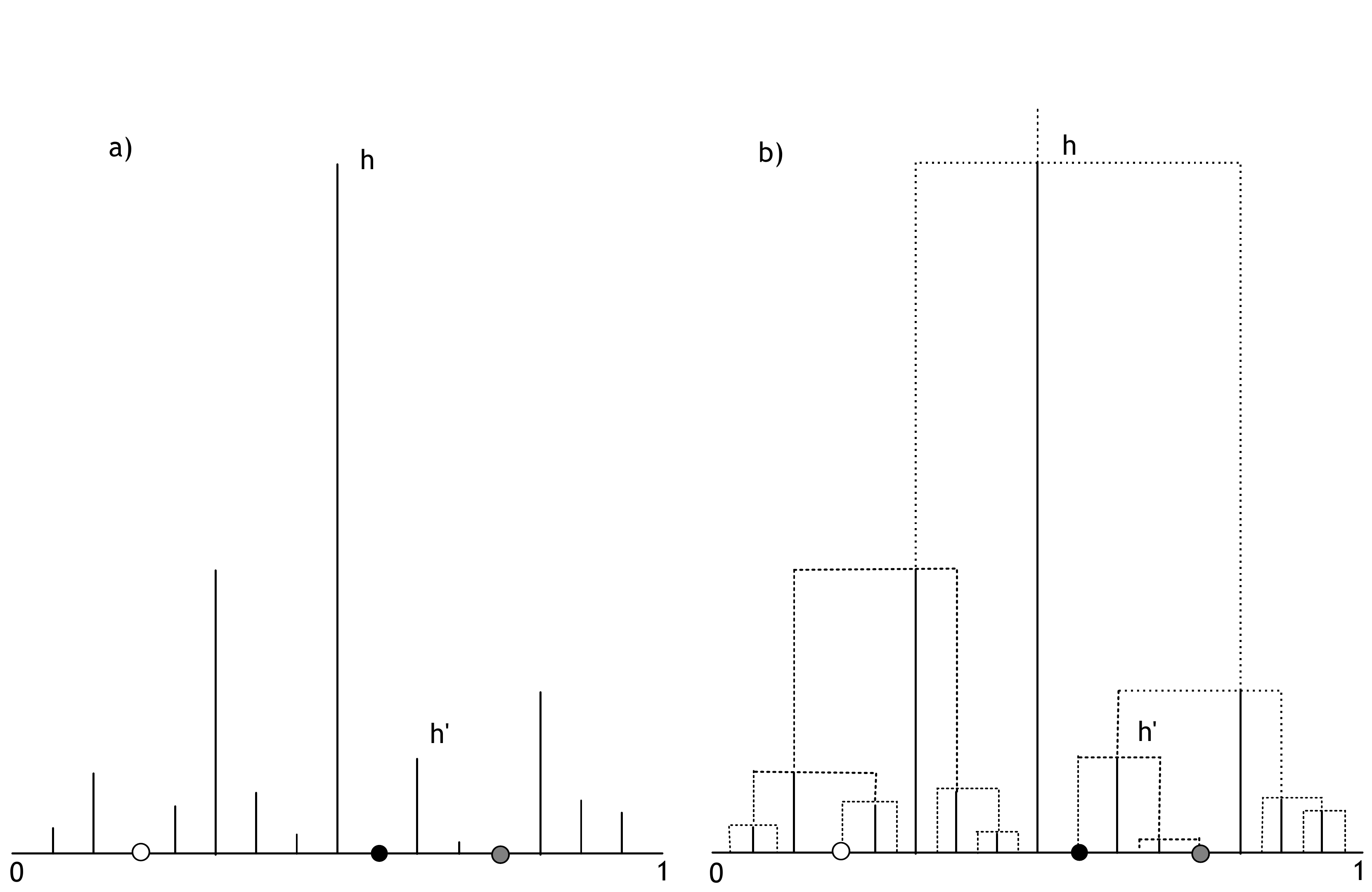}
\caption{a) A comb-like function with finite support on $[0,1]$. The distance between the black dot and the grey dot is $2h'$, whereas $2h$ is the distance from either of these dots to the white dot; b) In dashed lines, the ultrametric tree associated to the comb shown in a).}
\label{fig:comb}
\end{figure}
\noindent
Fig \ref{fig:comb} shows a comb and how an ultrametric tree can be embedded into it. \\

Let us give a first example of comb.  Let $S=\{x_i:1\le i\le n\}$ be a finite subset of elements of an ultrametric space $(U,d)$ labelled so as to satisfy the property expressed as Eq \eqref{eqn:intro} in the introduction. Set $h_i:=d(x_i,x_{i+1})$ for $1\le i\le n-1$, so that the metric of $S$ is completely characterized by these $n-1$ distances. Defining $I=[0,n]$ and $f:I\to\R_+$ by $f:=\sum_{i=1}^{n-1}h_i\indicbis{i}$, it is easy to see that $(S,d)$ is isometric to $(\dot I, \dtf)$. \\

Let us give an example of a random comb related to the celebrated Kingman coalescent. 
First, any comb can be related to an exchangeable coalescent as follows. Let $f$ be a comb on $[0,1]$ and $(V_i)$ be independent and identically distributed (i.i.d.) random variables uniform in $(0,1)$. Note that a.s. $f(V_i)=0$ for all $i$.  For any $t>0$, define  the partition $R_f(t)$ on $\N$ induced by the equivalence relation $\sim_t$
$$
i\sim_t j \Leftrightarrow \dtf(V_i, V_j) \le t.
$$
Note that if $\{f\not=0\}$ is dense, then $R_f(0+)$ is the finest partition (all singletons).
The process $(R_f(t);t>0)$ is an exchangeable coalescent process, in the sense that its law is invariant under permutations of $\N$ (see \citealt{BerBook2} for details).

Now let $f$ be itself random, defined on $I=[0,1]$ as
$$
f = \sum_{j\ge 1} \tau_j \indicbis{U_j},
$$
where the families of r.v. $(U_j)$ and $(\tau_j)$ are independent and independent of $(V_i)$, the $(U_j)$ are i.i.d. uniform on $(0,1)$ and $\tau_j=\sum_{k\ge j+1} e_k$, where $e_k$ are independent exponential r.v. with parameter $k(k-1)/2$. Then it is well-known that the process $(R_f(t); t\ge 0)$ has the same law as the Kingman coalescent (\citealt{Kin82}).

\subsection{Completion in the comb metric}
Now let us come back to a deterministic comb-like function $f$ on some compact interval $I$.
Consider $(s_n)$ a sequence of $\{f=0\}$ converging in the usual sense to $s\in I$ and set $I_n=(s\wedge s_n, s\vee s_n)$. Since for any $\varepsilon>0$, $\{f\ge \varepsilon\}\cap I_n$ is empty for $n$ large enough, $\dtf(s_n,s)$ goes to 0 as $n\to \infty$. 

If $f(s)=0$, and we still denote by $s\in \dot{I}$ its equivalence class, then $(s_n)$ converges to $s$ in $\dot{I}$ (i.e., in the comb metric). But if $f(s)\not=0$, then $(s_n)$ is a Cauchy sequence in $\{f=0\}$ for $\dtf$ without limit. Let us now describe how we perform the completion of $\dot{I}$.\\
\\
Let $t\in I$ such that $s<t$. If $(s_n)$ decreases, then $\dtf(s_n,t)$ converges to $\dtf(s,t)$. But if $(s_n)$ increases, then $\dtf(s_n,t)$ converges to $f(s)\vee\dtf(s,t)$. So when completing the comb metric space, we will need to distinguish between limits of increasing sequences and limits of decreasing sequences.\\ 
\\
Now we extend $\dtf$ to $I\times \{l,r\}$ by the following definitions for $s< t\in I$
$$
\dtf((s,r),(t,l)) = \dtf(s,t), \qquad \dtf((s,l),(t,l)) = f(s)\vee\dtf(s,t), 
$$
$$
\dtf((s,r),(t,r)) = \dtf(s,t)\vee f(t),  \qquad \dtf((s,l),(t,r)) = f(s)\vee\dtf(s,t)\vee f(t),
$$
or equivalently
$$
\dtf((s,r),(t,l)) = \sup_{(s, t)} f, \qquad \dtf((s,l),(t,l)) = \sup_{[s, t)} f, 
$$
$$
\dtf((s,r),(t,r)) =\sup_{(s, t]} f,  \qquad \dtf((s,l),(t,r)) = \sup_{[s, t]} f,
$$
and the symmetrized definitions for $s> t$. If $s=t$, the four last quantities are respectively defined as $f(t)$, 0, 0, $f(t)$.
We denote by $\bar{I}$ the quotient space $I\times\{l,r\}|_\sim$, where $\sim$ is the equivalence relation 
 $$
 s'\sim t'\Leftrightarrow \dtf(s',t')=0.
 $$
\begin{dfn}
For each $t\in I$, we will call $(t,l)$ the \emph{left face} of $t$ and $(t,r)$ its \emph{right face}. In particular if $f(t)=0$, the left face and the right face of $t$ are identified in $\bar{I}$.  
\end{dfn}

\begin{prop}
\label{prop:first sense}
The space $(\bar{I}, \dtf)$ is a compact, ultrametric space. 
\end{prop}
\noindent
The converse of this proposition will be the object of the next section.
\begin{dfn}
We call $(\bar{I}, \dtf)$ the \emph{comb metric space} associated with the comb $f$.
\end{dfn}
\begin{proof}[Proof of Proposition \ref{prop:first sense}] The fact that $\dtf$ is ultrametric is elementary. Let us show the compactness of $\bar I$.

 Let $(x_n')$ be a sequence of $\bar I$. Denote by $x_n$ the projection of $x_n'$ on $I$. Since $(x_n)$ takes values in the compact interval $I$, it has a converging subsequence in the usual metric. Up to extraction, we can assume that $(x_n)$ is a monotonic sequence converging in the usual metric to $x\in I$.
 
First assume the sequence $(x_n)$ has only finitely many elements which do not belong to $\{f=0\}$ (i.e., all but finitely many $x_n'$ are such that $x_n'=x_n$). If $x\in \{f=0\}$, then as explained in the beginning of this section, $(x_n')$ converges to $x$ in the comb metric, because $\dtf(x_n',x)= \dtf(x_n,x)=\sup_{(x_n\wedge x, x_n\vee x)}f$ which goes to 0 as $n\to\infty$. If $x\not\in\dot I$ and the subsequence is increasing (resp. decreasing), then there is convergence in the comb metric towards $x'=(x,l)$ (resp. $x'=(x,r)$). Indeed, assuming for example that $(x_n)$ is decreasing, then $\dtf(x', x_n')= \sup_{(x, x_n)}f$, which goes to 0 as $n\to\infty$. 
 
Now assume that $(x_n')$ contains infinitely many points belonging to $\bar I\setminus \dot I$, that is $(x_n)$ contains infinitely many points in $\{f\not=0\}$. Up to extraction, we can assume that that they are all left faces or all right faces. Without loss of generality, let us assume that they are all left faces. If $(x_n)$ is stationary, then $x_n'=(x,l)$ for $n$ large enough, and there is trivially convergence to $(x,l)$. Let us assume now that $(x_n)$ is not stationary so that $x_n\not=x$ for all $x$ (since $(x_n)$ is monotonic).

If $(x_n)$ is increasing, then $\dtf(x_n',(x,l))= \sup_{[x_n, x)}f$, which goes to 0 as $n\to\infty$, so that $(x_n')$ converges to $(x,l)$. If $(x_n)$ is decreasing then $\dtf((x,r),x_n')= \sup_{(x, x_n)}f$, which goes to 0 as $n\to\infty$ and $(x_n')$ converges to $(x,r)$. 
\end{proof}

\section{Compact ultrametric spaces are comb metric spaces}
\label{sec:converse}

\subsection{Isometry between ultrametric spaces and comb metric spaces}
In this section, we consider a compact ultrametric space $(U,d)$, endowed with a finite measure $\mu$ on its Borel $\sigma$-field. Set $m:=\mu(U)$ and $I=[0,m]$.

\begin{thm}
\label{thm:second sense}
Assume that $\mu$ has no atom and charges every ball with non-zero radius. Then there exists a comb-like function $f$ on $I$ and a map $\bar\phi:(\bar{I}, \dtf)\to (U,d)$, such that $\bar{\phi}$ is a global isometry, mapping the Lebesgue measure on $I$ to the measure $\mu$. 
\end{thm}

\begin{rem}
In the previous statement, it is implicit that the Lebesgue measure is extended to $\bar I$ by giving mass 0 to the countable set of points with distinct faces.
\end{rem}
\begin{proof}[Proof of Theorem \ref{thm:second sense}] 
For any $t>0$, let $\sim_t$ the relation on $U$ defined as
$$
x\sim_t y \Leftrightarrow d(x,y)\le t.
$$
The relation $\sim_t$ is clearly reflexive and symmetric. Let us show that it is also transitive. For any $x,y,z\in U$ such that $x\sim_t y$ and $y\sim_t z$, by ultrametricity of $d$, $d(x,z)\le \max(d(x,y), d(y,z)) \le t$, so that $x\sim_t z$. \\

Let $R(t)$ be the partition of $U$ into equivalence classes of $\sim_t$. For any $B\subset U$ element of $R(t)$, for any $x\in B$ and $y\in U$,
$$
y\in B \Leftrightarrow x\sim_t y\Leftrightarrow d(x,y)\le t,
$$ 
so that $B$ is actually the closed ball with center $x$ and radius $t$ (any element of a ball in an ultrametric space is its center). In particular, $\mu(B)>0$ and since $\mu$ is finite, the partition $R(t)$ is at most countable. Besides, the elements of $R(t)$ can be ranked in decreasing order of their measures so one can choose $x_n$ in its $n-$th element (countable axiom of choice). Let us show that this sequence cannot be infinite. Indeed, this would yield a sequence $(x_n)$ of $U$ such that $d(x_n,x_k)>t$ for any $n\not=k$, which would contradict the compactness of $U$. In conclusion $R(t)$ is finite for any $t>0$. Also note that $R(t)$ is reduced to the singleton $\{U\}$ as soon as $t\ge t_0:=\text{diam}(U)$.\\

Set $N(t):=\#R(t)$.
Observe that the process $(N(t);t>0)$ is a nonincreasing, right-continuous process with values in $\N$, and so it is càdlàg and its jump times can be labelled in decreasing order $t_0 > t_1 >\cdots$ Note that the process $(R(t); t>0)$ is also constant on any interval $[t_{n+1}, t_{n})$ and so is also càdlàg (for any topology on the set of partitions of $U$). 

As $t$ decreases, every ball of $U$ is thus fragmented into a finite number of balls with smaller, but positive, radius. Indeed, since $\mu$ charges all balls with positive radius and has no atom, $U$ has no isolated point and each ball can be partitioned into balls with strictly smaller, positive radii. In particular, the sequences of jump times of $R$ labelled in decreasing order, and of the trace of $R$ on any ball with positive radius, have limit 0.\\

Now we seek to order the elements of $R(t)$ in a consistent way. Namely define $R'(t):=U$ for any $t\ge t_0$, and more generally define $R'(t)$ as a $N(t)$-tuple of distinct elements of $R(t)$ recursively as follows. Let $n\ge 0$ and assume that $R'$ has been defined on $[t_n,+\infty)$. In particular, writing $k=N(t_n)$, $R'(t_n)$ is given in the form of a $k-$tuple $(B_1,\ldots, B_k)$.

Now for $t\in[t_{n+1}, t_n)$, replace each $B_i$ in this ordered list by the list of its subsets (which are all balls) with radius $t_{n+1}$ ranked (for example) in decreasing order of their measures. Notice that the elements of the resulting $R'(t)$ are not necessarily ranked in the order of their measures. \\

For any $t>0$, write $R'(t)=(B_1(t),\ldots, B_{N(t)}(t))$, set $A_0(t) =0$ and
$$
A_i(t) := \sum_{j=1}^i \mu(B_j(t))\qquad 1\le i \le N(t),
$$
so in particular $A_{N(t)}(t)= \mu(U)=m$. Also define
$$
J_i(t):= [A_{i-1}(t), A_i(t)] \qquad 1\le i \le N(t),
$$
and
$$
\Ac(t):=\{A_i(t); i=1,\ldots, N(t)-1\}.
$$
Then $(\Ac(t);t>0)$ forms a decreasing sequence of finite subsets of $(0,m)$ and we can define the countable set 
$$
\Ac:=\lim_{t\downarrow 0} \uparrow \Ac(t).
$$ 
As we pointed out earlier, the jump times of the trace of $(R(t))$ on any ball of $U$ has 0 as limit point, so that $\Ac$ is everywhere dense in $[0,m]$. Also observe that for any $a\in \Ac$, there is a unique $n\in \N$ such that for all $t>0$,
$$
a\in \Ac(t) \Leftrightarrow t\in (0,t_n).
$$
Set $f(a):=t_n$ and $f(x):=0$ for any $x\in[0,m]\setminus \Ac$. Then the function $f:[0,m]\to \R_+$ is a comb, because $\{f\ge \varepsilon\}$ is finite for any $\varepsilon >0$ (also note that here $\{f\not=0\}$ is dense so there is no need to consider the quotient space of $\dot I =\{f=0\}$).\\

Let $r\in[0,m]$ such that $f(r)=0$. Since $r\not\in \Ac$, for any $t>0$, there is a unique integer $n_r(t)$ such that $r\in J_{n_r(t)}(t)$. And since by construction (by definition of the $B_i$'s), the intervals $J_i(t)$ are fragmented as $t$ decreases, the sequence $J_{n_r(t)}(t)$ decreases (in the inclusion sense) as $t$ decreases. Note that the corresponding balls of $U$, $B_{n_r(t)}(t)$ satisfy $\mu(B_{n_r(t)}(t)) = |J_{n_r(t)}(t)|$, and also decrease in the inclusion sense as $t$ decreases. By definition, the radius of $B_{n_r(t)}(t)$ is $t$, and so this ball is nonempty. Therefore, the $B_{n_r(t)}(t)$ form a decreasing family (as $t$ decreases) of nonempty compact balls with radius $t$. By the finite intersection property, the limit of this sequence as $t\downarrow 0$ is nonempty, and since its radii go to 0, the limit is reduced to a singleton that we will denote $\phi(r)$.

Notice that the mapping $\phi$ satisfies for all $r\in \{f=0\}$
$$
r\in J_i(t)\Leftrightarrow \phi(r)\in B_i(t)\qquad t>0, 1\le i\le N(t).
$$
This shows in particular that the pre-image of $B_i(t)$ by $\phi$ is $J_i(t)$, which has Lebesgue measure $\mu(B_i(t))$ by construction. In other words, $\mu$ and the push-forward of Lebesgue measure by $\phi$ coincide on closed balls of $U$, and so coincide on the Borel $\sigma$-field of $U$.\\
 
Now for any $0\le r<s\le m$ such that $f(r)=f(s)=0$, for any $t>0$,
\begin{multline*}
\dtf(r,s)\le t \Leftrightarrow \sup_{(r,s)} f\le t \Leftrightarrow \Ac(t)\cap(r,s)=\varnothing \\ \Leftrightarrow \exists i: r,s\in J_i(t)\Leftrightarrow\exists i : \phi(r),\phi(s)\in B_i(t)\Leftrightarrow d(\phi(r),\phi(s))\le t,
\end{multline*}
so that $\phi:(\{f=0\},\dtf) \longrightarrow (U,d)$ is an isometry.\\

Recall that $I=[0,m]$, so that $\dot I=\{f=0\}= I\setminus \Ac$ and $\bar I$ is the completion of $\dot{I}$ embedded in $I\times\{l,r\}$ as in Proposition \ref{prop:first sense}. Then $\dot{I}$ is of course dense in $\bar{I}$ for the distance $\dtf$ and it is standard that there exists a unique continuous extension $\bar\phi:(\bar{I},\dtf)\to (U,d)$ of $\phi$, which is obviously an isometry.\\

It only remains to prove that $\bar{\phi}$ is surjective. Set $F:=\bar\phi(\bar I)\subset U$. Since $\mu(F) = \text{Leb}(I)=m =\mu(U)$, we have $\mu(U\setminus F)=0$ so that  $U\setminus F$ contains no ball of positive radius. Assume that there is $y\in U\setminus F$. For any $\varepsilon>0$, the ball centered at  $y$ with radius $\varepsilon$ is not contained in $U\setminus F$, and so it intersects $F$. This shows that $y$ is the limit of a sequence $(y_n)$ of elements of $F$. So for each $n$, there is $x_n\in \bar I$ such that $y_n=\bar\phi(x_n)$. Since $\bar\phi$ is isometric and $(y_n)$ is a Cauchy sequence, $(x_n)$ also is one, and thereby converges to some $x\in \bar I$. By continuity of $\bar\phi$, $\bar\phi(x) = y$.
\end{proof}
\begin{rem}
Any compact, ultrametric space $(U,d)$ can be endowed with a finite measure charging every ball with non-zero radius. One example of such a measure is the so-called \emph{visibility measure} (\citealt{Lyo94}), giving mass 1 to $U$, and where the mass is equally divided between sub-balls at each jump time of the partition $R$ defined in the beginning of the proof. Since the atoms of the visibility measure are the isolated points of $U$, the theorem ensures that \textbf{any compact ultrametric space without isolated point is isometric to a comb metric space.} 
\end{rem}

\subsection{Two beautiful examples}

In this section, we consider two well-known (locally) compact ultrametric spaces and display for each one an isometric embedding into a comb metric space. More specifically, we fix an integer $p\ge 2$ and we introduce the following two ultrametric spaces. 

First, we consider the set $U_p$ of sequences $x=(x_n)_{n\in\Z}$ indexed by $\Z$ with values in $\{0,\ldots,p-1\}$ for which there is $n\in\Z$ such that $x_k=0$ for all $k\le n$, seen as the boundary of the (doubly) infinite $p$-ary tree, endowed with the distance $d_u$ defined below. In a second paragraph, we will consider the field $\Q_p$ of $p$-adic numbers endowed with the $p$-adic metric. For other works concerning graphic representations of $\Q_p$, see also \cite{Cuo91, Hol01}.

\subsubsection{The boundary of the infinite $p$-ary tree}

The space $U_p$ is endowed with the distance $d_u$ defined for any pair $x=(x_n)_{n\in\Z}$ and $y=(y_n)_{n\in\Z}$ of elements of $U_p$ by
$$
d_u(x,y) = p^{-v_u(x-y)},
$$
where 
$$
v_u(x):=\min\{n\in\Z:x_n\not=0\}, 
$$
with the convention $p^{-v_u(0)}=0$.  
It is well-known that the space $(U_p,d_u)$ is a locally compact, ultrametric space. Also note that its ultrametric structure can be seen directly by mapping $U_p$ to the boundary of the infinite $p$-ary tree. We propose here an alternative embedding, of the comb type, whereby left faces correspond to the infinite lines of descent in the $p$-ary tree.

Actually, we show how to construct explicitly the global isometry of the theorem between $(U_p,d_u)$ and a comb metric space defined on the whole half-line. The fact that the sequences are indexed by $\Z$ is actually not a requirement for this first example (but it will be for the next one), and the reader might as well think of sequences as indexed by $\N$ (whereby $(U_p,d_u)$ becomes compact).

Let $U'$ be the subset of $U_p$ consisting of sequences $x$ such that $\lim_{n\to+\infty}x_n =0$. For any nonzero $x\in U'$, set
$$
w(x):=\max\{n\in \Z :x_{n}\not=0\},
$$
and define $\hat{x}\in U_p$ by
$$
\hat{x}_n:=
\left\{
\begin{array}{cl}
x_n &\mbox{ if } n<w(x)\\
x_n-1 &\mbox{ if } n=w(x)\\
p-1 &\mbox{ if } n>w(x).
\end{array}
\right.
$$
Note that $v_u(x-\hat{x})=w(x)$ so that $d_u(x,\hat{x})=p^{-w(x)}$.\\
 
Now recall that any non-negative real number $t$ can be mapped to a unique $x\in U_p$ and a unique nondecreasing sequence $(t_n)$ such that $t_n=\sum_{k\le n}  p^{-k} x_k$ and $t_n\le t < t_n + p^{-n}$,
so that in particular 
$$
t=\sum_{k\in\Z}  p^{-k} x_k.
$$
Let us write
$$
\phi(t):=x.
$$
Also if $x=\phi(t)\in U'$, 
we define $\phi_r(t)=x$ and $\phi_l(t)=\hat{x}$.
For example, with $p=2$, $\phi_r(3/4)=(\ldots,0,0;0,1,1,0,0,0\ldots)$ and $\phi_l(3/4)=(\ldots,0,0;0,1,0,1,1,1\ldots)$ (the semi-colon separates the negative indices from the non-negative indices). 

It is immediate to see that $\phi$ is one-to-one, with left-inverse $\phi^{-1}$ mapping $x\in U_p$ to $\sum_{k\in\Z}  p^{-k} x_k$. Note that $\phi^{-1}$ itself is not one-to-one since for any $x\in U'$, $\phi^{-1}(x)=\phi^{-1}\left(\hat{x}\right)$. Indeed, denoting $n=w(x)$ and $t=\phi^{-1}(x)=\sum_{k\le n}  p^{-k} x_k$, we get
$$
\sum_{k\in\Z}  p^{-k} \hat{x}_k = \sum_{k<n}  p^{-k} x_k + p^{-n}(x_n-1) +\sum_{k>n}  p^{-k} (p-1) = t -p^{-n}+(p-1)\frac{p^{-n-1}}{1-1/p}= t.
$$
Finally, for any $t\in\R_+$, define
$$
F_p(t):=
\left\{
\begin{array}{cl}
p^{-w(\phi(t))} &\mbox{ if }  \phi(t)\in U'\\
0 &\mbox{ otherwise.}
\end{array}
\right.
$$
The function $F_p$ is a comb on each compact subinterval of $I=\R_+$, and it is not difficult to see that $\dtfp$ actually defines a comb metric space $(\bar{I},\dtfp)$ which is not compact, but which is locally compact. Then we can construct a global isometry $\bar{\phi}:(\bar{I},\dtfp)\to(U_p,d_u)$ as follows.
$$
\bar{\phi}(t'):=
\left\{
\begin{array}{cl}
\phi(t) &\mbox{ if } t'=t \in \{F_p=0\}\\
\phi_l(t) &\mbox{ if } t' =(t,l)\\
\phi_r(t)& \mbox{ if } t' =(t,r).
\end{array}
\right.
$$
It is an exercise to show that the isometry of the theorem can be taken equal to the function $\bar{\phi}$ just defined when $\mu$ is the visibility measure adapted to the locally compact case, that is the measure which puts mass $p^{n}$ on any ball with radius $p^n$. \\

\begin{figure}[!ht]
\includegraphics[width=\textwidth]{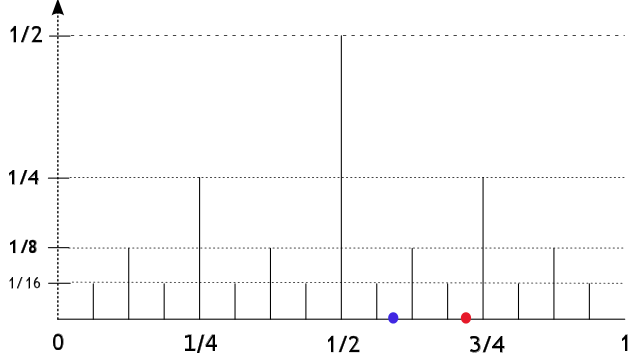}
\caption{The dyadic comb $F_2$ on $[0,1]$. The set of sequences of $0$'s and $1$'s indexed by $\N$ can be mapped isometrically by $\bar\phi^{-1}$ to the comb metric space associated with $F_2$ on $[0,1]$. The right face (resp. left face) of each element of $\{F_2\not=0\}$ is the image of a sequence with limit 0 (resp. with limit 1).  
In the comb metric the two dots are at distance $1/8$. On the other hand, $\phi$ maps the blue (left) dot to some sequence $(\ldots,0,0;0,1,0,0,1,\ldots)$ and the red (right) dot to some sequence $(\ldots,0,0;0,1,0,1,1,\ldots)$, so the distance between these two sequences is $2^{-3}=1/8$, as expected. }
\label{fig:dyadiccomb}
\end{figure}

\subsubsection{The $p$-adic (ultra)metric}

 We now deal with another locally compact ultrametric space, namely the field $\Q_p$ of $p$-adic numbers, where  $p$ is now assumed to be a prime number. For any nonzero integer $a$, let $v_p(a)$ denote the power of $p$ in the decomposition of $a$ in product of prime numbers, and for any nonzero rational number $q=a/b$ set $v_p(q):=v_p(a)-v_p(b)$, which does not depend on the ratio of integers $a$ and $b$ chosen to represent $q$. The $p$-adic distance $d_p$ between two rational numbers $q$ and $q'$ is defined as
$$
d_p(q,q') = p^{-v_p(q-q')},
$$
with the convention $p^{-v_p(0)}=0$. 

The set $\Q_p$ called field of $p$-adic numbers is the completion w.r.t. the metric $d_p$ of the set of $p$-adic numbers. By definition, a $p$-adic number $q$ is a rational number that can be written in the form 
\begin{equation}
\label{eqn:rational}
q=\sum_{k\in \Z} p^{-k} x_k,
\end{equation}
where $x$ is a sequence with values in $\{0,\ldots, p-1\}$ and finitely many nonzero terms. This (unique) sequence $x$ was denoted $\phi(q)$ in the previous paragraph but here we will write 
$$
x:=\psi(q),
$$
because we aim at extending $\psi$ to $\Q_p$. Notice that 
$$
v_p(q) = -\max\{k\in\Z: x_{k}\not=0\},
$$ 
so if $y=\rho(x)$ denotes the sequence $(x_{-k})_{k\in\Z}$, then $v_p(q)= \min\{k\in\Z: y_{k}\not=0\} =v_u(y)$, that is
$$
v_p(q) = v_u(\rho\circ\psi(q)).
$$
Therefore,
$$
d_u(0,\rho\circ\psi(q)) = p^{-v_u(\rho\circ\psi(q))}=p^{-v_p(q)} = d_p(0, q).
$$
Due to the linearity of $\psi$ (in the $p$-adic addition) and of $\rho$, this equality extends to any pair of rational numbers $q,q'$
$$
d_u(\rho\circ\psi(q'),\rho\circ\psi(q)) = d_p(q, q').
$$
For now, the map $\rho\circ\psi$ establishes a local isometry from $(\Q, d_p)$ to $(U_p, d_u)$.
By definition, $\Q_p$ is the completion of $\Q$ with respect to the $p$-adic metric $d_p$, and it is known that any $q\in \Q_p$ can be written in a unique way in the form \eqref{eqn:rational} called Hensel's decomposition, where $x$ is still denoted $\psi(q)$ but can now have infinitely many nonzero terms indexed by negative indices and the sum is understood in the $p$-adic sense. 
As a consequence, we get the following proposition.
\begin{prop} The function $\rho\circ\psi:(\Q_p, d_p)\to (U_p, d_u)$ is a global isometry. As a consequence, the function $\bar{\phi}^{-1}\circ\rho\circ\psi:(\Q_p, d_p)\to (\bar{I}, \dtfp)$ is a global isometry between $\Q_p$ and a comb metric space.
\end{prop} 

\begin{figure}[!ht]
\includegraphics[width=\textwidth]{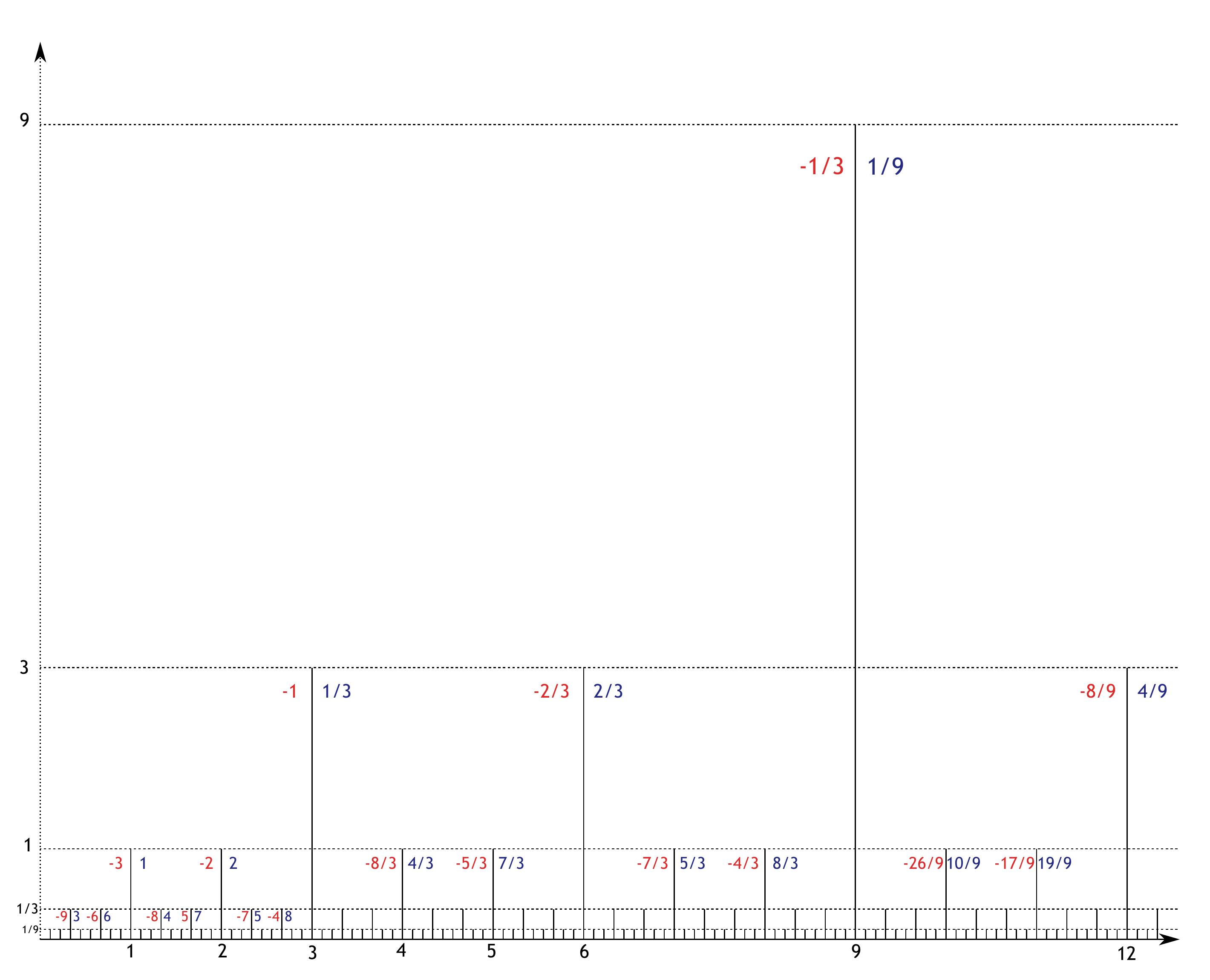}
\caption{The triadic comb $F_3$ on $[0,\infty)$. The field $\Q_3$ can be mapped isometrically by $\chi$ to the comb metric space associated with $F_3$. Each triadic number is mapped by $\chi$ to a left face (in red) or a right face (in blue) of $\{F_3\not=0\}$. The limits of Cauchy sequences of triadic numbers are mapped by $\chi$ to $\{F_3=0\}$. Check on the figure that $d_3(4/9,1/3)=d_3(4/9, 2/3)=9$, that $d_3(1/3,-1/3)=d_3(1/3, 2/3)=3$ and that $d_3(-2,2) = d_3(-2,5)=1$. 
 }
\label{fig:triadicfield}
\end{figure}

\noindent
Let us give a few examples. We set $\chi:=\bar{\phi}^{-1}\circ\rho\circ\psi$ so that $\chi^{-1}={\psi}^{-1}\circ\rho\circ\bar{\phi}$ ($\rho$ is an involution). 
We will use repeatedly the following fact. 
\begin{lem}
Let $n\in \Z$ and $x^{(n)}\in U$ be the sequence defined by $x_k^{(n)}=0$ if $k<n$ and $x^{(n)}_k =p-1$ for all $k\ge n$. Then $\psi^{-1}\circ\rho\left(x^{(n)}\right)\in \Q_p$ is actually the integer $-p^{n}$. 
 \end{lem}
\begin{proof}
First, $\rho\circ \psi(p^{n})$ is the sequence say $a^{(n)}$, with only zeros except the term at rank $n$, equal to 1. Then using the fact that $\psi^{-1}\circ\rho$ is a homomorphism, we get
$$
\psi^{-1}\circ\rho\left(x^{(n)}\right)+p^{n} = \psi^{-1}\left(\rho\left(x^{(n)}\right)+\rho\left(a^{(n)}\right)\right).
$$
Now the previous $p$-adic sum also equals $\psi^{-1}\left(\rho\left(x^{(n+1)}\right)+\rho\left(a^{(n+1)}\right)\right)$, and so equals $\psi^{-1}\left(\rho\left(x^{(k)}\right)+\rho\left(a^{(k)}\right)\right)$ for any $k\ge n$, whose $p$-adic distance to 0 vanishes as $k\to\infty$.
 \end{proof} 
\noindent
For the following examples, we fix $p=3$ and we refer to Figure \ref{fig:triadicfield}.

Let $q=1=1p^0$, so that $\rho\circ\psi(q)=\psi(q)=(\ldots 0,0,0;1,0\ldots)$ and $\chi(q)$ is equal to $(1,r)$, the right face of 1. To know which element of $\Q_p$ is mapped to the left face of $1$ by $\chi$, we look at $\chi^{-1}((1,l))$. First, $\bar{\phi}((1,l)) =\phi_l(1)= (\ldots 0;0,2,2,2\ldots)$, so by the previous lemma, the image by $\psi^{-1}\circ\rho$ of this sequence is $-3$, that is, $\chi^{-1}((1,l))=-3$.

Let $q=6=2p^1$, so that $\psi(q)=(\ldots 0,0,2;0,0\ldots)$ and $\rho\circ\psi(q)=(\ldots 0,0;0,2,0\ldots)$, so that $\chi(q)=(2/3,r)$, the right face of $2/3$. Let us consider $\chi^{-1}((2/3,l))$. First, $\bar{\phi}((2/3,l)) =\phi_l(2/3)= (\ldots 0;0,1,2,2\ldots)$, whose image by $\psi^{-1}\circ\rho$ is $\psi^{-1}(\ldots0,0, 1;0,0\ldots) +\psi^{-1}\circ\rho (\ldots 0;0,0,2,2\ldots)=3-3^2=-6$.

We just saw that $\chi^{-1}((1,l))=-3$ whereas $\chi^{-1}((1,r))=1$, and that $\chi^{-1}((2/3,l))=-6$ whereas $\chi^{-1}((2/3,r))=6$. Actually, it is a simple exercise to show that for any $p$-adic number $q$ such that $F_p(\chi(q))=p^{-n}$, we have
$$
\chi^{-1}((q,l)) - \chi^{-1}((q,r)) = -(p+1)p^n,
$$
which gives the difference between the images of each face of the same point. Check this on Fig \ref{fig:triadicfield}.

\section{Spheres of real trees}
\label{sec:spheres}

\subsection{Preliminaries on real trees}

A \emph{real tree} is a complete, path-connected metric space $(\tr, d)$ satisfying the so-called `four points condition': for any $x_1, x_2, x_3, x_4\in\tr$
\begin{equation}
\label{eqn:4 points}
d(x_1,x_2)+d(x_3,x_4)\le \max\left[d(x_1,x_3)+d(x_2,x_4), d(x_1,x_4)+d(x_2,x_3)\right] .
\end{equation}
The root $\rho$ is a distinguished element of $\tr$. Real trees have been studied extensively, see e.g.  \citet{DMT96, EvaBook}. For example, it is well-known that real trees have unique geodesics and are endowed with a partial order $\preceq$. For any $x,y\in\tr$, $x\preceq y$ iff $x$ belongs to the geodesic between $\rho$ and $y$. \\

Also, real maps can be used to define trees. More specifically, let $h:[0,\infty)\to[0,\infty)$ be càdlàg (right-continuous with left-hand limits), with no negative jump and with compact support. Set $\sigma_h:=\sup\{t>0: h(t)\not=0\}$, as well as
$$
m_h(s,t):=\inf_{[s\wedge t, s\vee t]}h\qquad s,t\ge 0,
$$
and
$$
d_h(s,t):=h(s)+h(t) -2m_h(s,t).
$$
It is clear that $d_h$ is a pseudo-distance on $[0,\infty)$. Further let $\sim_h$ denote the equivalence relation on $[0,\infty)$
$$
s\sim_h t \Leftrightarrow d_h(s,t)=0 \Leftrightarrow h(s)=h(t)=m_h(s,t).
$$
If $\tr_h$ denotes the quotient space $[0,\sigma_h]|_{\sim_h}$, then it is well-known that $(\tr_h,d_h)$ is a compact real tree.
Let $p_h:[0,\sigma_h]\to \tr_h$ map any element of $[0,\sigma_h]$ to its equivalence class relative to $\sim_h$. We can also endow $\tr_h$ with a total order and a mass measure, as follows. 
\begin{itemize}
\item Total order. We define $\le_h$ as the order of first visits, that is for any $x,y\in\tr_h$,
$$
x\le_h y \Leftrightarrow \inf p_h^{-1}(\{x\})\le \inf p_h^{-1}(\{y\}).
$$
\item Mass measure. The measure $\mu_h$ is defined as the push forward of Lebesgue measure by $p_h$.
\end{itemize}
\noindent
Conversely, let $\tr$ be a compact, real tree endowed with a total order $\le$ and a finite measure $\mu$.  \cite{LUB15} have provided some simple conditions on $\le$ and $\mu$ under which $\tr$ can be viewed as a tree coded by some real map. First, the order $\le$ is assumed to satisfy the following two conditions. For any $x,y\in\tr$,
\begin{itemize}
\item
If $x\preceq y$ then $y\le x$
\item If $x\le y$ then for all $z\in\tr$ 
$$
z\preceq x \mbox{ and } z\not\preceq y\Rightarrow z\le y.
$$
\end{itemize}
Second, the measure $\mu$ is assumed to have no atom and to charge all open subsets. Then the map $\varphi:\tr \to[0,\mu(\tr)]$ defined by $\varphi(x)=\mu(\{y\in\tr:y\le x\})$ is one-to-one and its image is dense in $[0,\mu(\tr)]$. Its inverse has a unique càdlàg extension, denoted $\phi:[0,\mu(\tr)]\to\tr$. Furthermore, the map $h:[0,\mu(\tr)]\to[0,\infty)$ defined by $h(s) = d(\rho,\phi(s))$ is càdlàg with no negative jump and the tree $(\tr, \le, \mu)$ is isometric to $(\tr_h, \le_h, \mu_h)$. The map $h$ is called the \emph{jumping contour process} of $\tr$.\\

From now on, we let $\tr$ denote a compact real tree with total order $\le$ and finite measure $\mu$ satisfying the previous conditions.
We fix a real number $T>0$, and we let \emph{\bf $h$ denote the  jumping contour process of the tree $\tr$ truncated at height $T$}, which is the closed ball with center $\rho$ and radius $T$. We denote by $\trT$ the sphere with center $\rho$ and radius $T>0$  
$$
\trT:= \{x\in\tr : d(\rho,x)=T\},
$$
which we assume to be nonempty.
The so-called \emph{reduced tree} at height $T$ is the tree spanned by the sphere of radius $T$, that is 
$$
\{y\in \tr: \exists x\in \trT, y\preceq x\} ,
$$
which is the union of all geodesics linking the root to a point of the sphere. 

The reduced tree is also called \emph{reconstructed tree} in phylogenetics and \emph{coalescent tree} in population genetics. In these fields, the sphere and the reduced tree play an important role, since the former represents the set of individuals or species living at the same time and the latter represents their genealogical history (see \citealt{Lam16}). Actually, the induced metric on the sphere of radius $T$ contains all the information about the ancestral relationships between them. We will now try to understand how this metric can be represented by a comb.
Indeed, note that by the four-points condition, for any $x,y,z\in\trT$,
\begin{multline*}
T+d(x,z)=\\d(\rho,y)+d(x,z)\le \max\left[d(\rho,x)+d(y,z), d(\rho,z)+d(y,x)\right]\\ =\max \left[T+d(y,z), T+d(y,x)\right],
\end{multline*}
which yields $d(x,z)\le \max \left[d(y,z), d(y,x)\right]$, so that the metric induced by $d$ on $\trT$ is ultrametric.

\subsection{Sphere of a real tree}

The sphere $\trT$ is a compact ultrametric space so Theorem \ref{thm:second sense} ensures that provided it has no isolated point, it is isometric to a comb metric space. Instead of relying on this theorem, we prefer to give a direct construction of this comb by taking advantage of the order on $\trT$ inherited from the order $\le$ on $\tr$.\\

In the case when $\trT$ has only isolated points, it is not difficult to understand its metric structure. For example, assume that $\trT$ has finite cardinality $N_T\ge 2$. Let $x_1\le \cdots \le x_{N_T}$ denote its elements labelled in the order $\le$. Then for any $1\le i <j\le N_T$, writing $s_i:=\inf p_h^{-1}(\{x_i\})$ (each set $p_h^{-1}(\{x\})$ is actually a singleton 
in the case when the sphere is finite),
$$
d(x_i,x_{j})= h(s_i) + h(s_{j})- 2 \inf_{[s_i,s_{j}]}h = 2(T- \inf_{[s_i,s_{j}]}h)= 2\max(h_i,\ldots, h_{j-1}),
$$
 where
$$
h_i:=T- \inf_{[s_i,s_{i+1}]}h.
$$
In conclusion, the metric on $\trT$ is isomorphic to the comb metric $\dtf$ on $\dot I$, where $I=[1,N_T]$ and $f:=2\sum_{i=1}^{N_T-1}h_i\indicbis{i}$. We will now extend this description to spheres with no isolated points, but we first need the following lemma.

\begin{lem}
Assume that $\trT$ has no isolated point. Then $\{h=T\}$ has no isolated point and empty interior. 
\end{lem}
\begin{proof}
Assume that $t$ is an isolated point of $\{h=T\}$. Set $G:=\sup\{s<t:h(s)=T\}$ and $D:=\inf\{s>t:h(s)=T\}$. By assumption, $G<t<D$ and since $h$ is càdlàg, $h(G-) =T$ and $h(D)=T$. Set $\gamma:=\inf_{[G,t]} h$ and $\delta:=\inf_{[t, D]} h$. Recall that $h$ is the contour process of the tree truncated below $T$, so that $h$ only takes values in $[0,T]$. This implies that $\gamma<T$ and $\delta<T$. So for any $s$ such that $h(s)=T$, $d_h(s,t)>T-(\gamma\vee\delta)$, so that $p_h(t)$ is an element of $\trT$ which is at distance at least $T-(\gamma\vee\delta)$ of all other points of $\trT$ (which all are of the form $p_h(s)$ for $s\in\{h=T\}$). This contradicts the assumption that $\trT$ has no isolated point. 

Now if $\{h=T\}$ had nonempty interior, there would be an open interval $(s,t)$ such that $h(u)=T$ for all $u\in(s,t)$. Then for any $u,v\in(s,t)$, $p_h(u)=p_h(v)=:x\in \trT$, so that $\mu(\{x\})\ge t-s>0$, which contradicts the fact that $\mu$ has no atom.
\end{proof}
\noindent
Thanks to the previous result, $\{h=T\}$ is perfect, so we can construct a local time at level $T$ for $h$, that is a nondecreasing, continuous map $L:[0,\infty)\to [0,\infty)$ such that $L(0)=0$ and for any $s<t$
$$
L_t>L_s \Leftrightarrow (s,t)\cap\{h=T\}\not=\varnothing.
$$
The reader who would like to see a detailed construction of this map is deferred to Section \ref{sec:local time} in the appendix.

Let $I=[0,L_{\sigma_h}]$, and set
$$
\dottrT:= \{x\in\trT : \exists (x_n^+)\uparrow, (x_n^-) \downarrow\in\trT, \lim_n\uparrow x_n^+=\lim_n\downarrow x_n^-=x\},
$$
where the sequences in the previous definition are requested to be strictly monotonic.
\begin{thm} 
\label{thm:comb-reducedtree}
Assume as previously that $\tr$ is a compact real tree such that $\trT$ is not empty and has no isolated point. Then there is a comb-like function $f$ on $I$ and two global isometries $\dot\theta:(\dot I, \dtf)\to(\dottrT,d)$ and $\bar\theta:(\bar I, \dtf)\to(\trT,d)$ preserving the order and mapping the Lebesgue measure to the push forward  $\mu^T$ of the Lebesgue--Stieltjes measure $dL$ by $p_h$.
\end{thm}
\begin{rem} Note that we could have equally assumed that $\tr$ is only locally compact, since then its closed balls are compact by the Hopf-Rinow theorem. 
\end{rem}
\begin{proof}
Let $J$ denote the right inverse of the local time $L$
$$
J_t := \inf\{s>0:L_s >t\}\qquad t\in I.
$$
Every jump time $s$ of $J$ corresponds to an interval $(g_s,d_s):=(J_{s-}, J_s)$ where $L$ is constant (to $s$) such that $h\not=T$ on $(g_s,d_s)$ and $h(g_s-)=h(g_s)= T = h(d_s)$ (because $h$ has no negative jump). Note that $J$ is càdlàg so that the number of jumps with size larger than $\varepsilon$ is finite. Now for any $s\in I$, set
$$
f(s):=
\left\{
\begin{array}{cl}
2\,(T-\inf_{[g_s, d_s]}h)&\text{ if }\Delta J_s \not=0\\
0 &\text{ otherwise,}
\end{array}
\right.
$$
which is a comb on $I$, so we can consider $(\dot I, \dtf)$ and $(\bar I, \dtf)$ the associated ultrametric spaces. 
Note that $f(s)=0$ implies that $\Delta J_s =0$ otherwise $h$ would be constant to $T$ on $[g_s,d_s]$ (and so the mass measure would have an atom).  

Further define
\Fonction{\dot\theta}{(\dot I,\dtf)}{(\dottrT, d)}{s}{p_h(J_s)}
as well as
\Fonction{\bar\theta}{(\bar I,\dtf)}{(\trT, d)}{s'}{
\left\{
\begin{array}{cl}
p_h(J_{s-})& \text{ if $s'=(s,l)$ and $f(s)\not=0$}\\
 p_h(J_s)& \text{ otherwise.} 
\end{array}
\right.
}
\noindent
Let us first prove that $\dot\theta$ and $\bar\theta$ do take their values in $\dottrT$ and $\trT$ respectively.

Let $s\in\dot I$, so that $f(s)=0$ and $\Delta J_s =0$, as noted earlier. So $t:=J_s$ is a point of increase of $L$. Because $L$ is continuous and nondecreasing, and increases only on $\{h=T\}$, $t$ is both a left and a right limit point of $\{h=T\}$.  So there is an increasing sequence $(t_n^+)$ and a decreasing sequence  $(t_n^-)$ of $\{h=T\}$ both converging to $t$. Since $h$ is càdlàg, $\inf_{[t_n^+, t_n^-]}h$ converges to $T$, so that $p_h(t_n^+)$ and $p_h(t_n^-)$ are resp. increasing and decreasing sequences of $\trT$ both converging to $p_h(t)=p_h(J_s)= \dot\theta (s)$. This shows that $\dot\theta(s)\in\dottrT$.

The case of $\bar\theta$ is simpler, since $h(g_s)= h(d_s)=T$ for all $s$, so that $p_h(J_{s-})$ and $p_h(J_s)$ are always elements of $\trT$.\\
\\
Let us now prove that both $\dot\theta$ and $\bar\theta$ are surjective. In both cases, we will use the fact that for any $x\in\trT$, there is a unique $t$ such that $p_h(t)=x$ (otherwise $h$ would be constant to $T$ on some open interval). And since $h(t)=T$, there is $s$ such that $t=J_s$ or $t=J_{s-}$.

Let us first treat the case of $\bar\theta$. If $f(s)=0$ then $J_s=J_{s-}$ and $x=p_h(J_s)=\bar\theta(s)$. If $f(s)\not=0$, then either $t=J_s$ and $x=p_h(J_s)=\bar\theta((s,r))$ or $t=J_{s-}$ and $x=p_h(J_{s-})=\bar\theta((s,l))$. So $\bar\theta$ is surjective.

Now let $x\in \dottrT$ and let $(x_n^-)$ and $(x_n^+)$ be resp. decreasing and increasing sequences of $\trT$ converging to $x$. Then for each $n$ there is a unique pair $(t_n^-,t_n^+)$ such that $p_h(t_n^-) =x_n^-$ and $p_h(t_n^+) =x_n^+$. So $(t_n^-)$ and $(t_n^+)$ are resp. decreasing and increasing sequences of $\{h=T\}$, and both converge to $t$ (otherwise $h$ would be constant to $T$ on some open interval). So $t$ is a point of increase of $L$ and $t=J_s = J_{s-}$. This implies that $f(s)=0$ and $x=p_h(t) = p_h(J_s) = \dot\theta(s)$. So $\dot\theta$ also is surjective.\\
\\
Next, we prove that  $\dot\theta$ and $\bar\theta$ are isometric. For any $s<t\in \dot I$,
\begin{eqnarray*}
d( \dot\theta(s),\dot\theta(t)) &=& d(p_h(J_s), p_h(J_t)) \\
	&=& h(J_s) + h(J_t)- 2 \inf_{[J_s, J_t]}h\\
		&= & 2T - 2\inf_{[J_s, J_t]}h = 2\sup_{[J_s, J_t]}(T-h)\\
		&=& 2\max\left\{	\sup_{[J_s, J_t]\setminus \cup_u [g_u,d_u]}(T-h), \sup_{ u\in[s,t] }\sup_{[g_u,d_u]}(T-h)	\right\}\\
		&=& 2\sup_{ u\in[s,t] }\sup_{[g_u,d_u]}(T-h) =2\sup_{u\in[s,t]}\left(T-\inf_{[g_u,d_u]}h\right)\\
		&=& \sup_{u\in[s,t]}f(u)= \dtf(s,t).
\end{eqnarray*}
Now for any $s'<t'\in\bar I$, according to whether $s'=(s,l)$ or $s'=(s,r)$,  $t'=(t,l)$ or $t'=(t,r)$,
\begin{eqnarray*}
d( \bar\theta(s'),\bar\theta(t')) &=& d(p_h(J_{s\pm}), p_h(J_{t\pm})) \\
	&=&2\max\left\{	\sup_{[g_s,d_s]}(T-h)\indic{s'=(s,l)}, \sup_{ u\in(s,t) }\sup_{[g_u,d_u]}(T-h),\sup_{[g_t,d_t]}(T-h)\indic{t'=(t,r)} 	\right\}\\
	&=&\max\left\{	f(s)\indic{s'=(s,l)}, \sup_{ u\in(s,t) }f(u),f(t)\indic{t'=(t,r)} 	\right\}\\
		&=&  \dtf(s',t').
\end{eqnarray*}
\noindent The fact that $\dot\theta$ and $\bar\theta$ preserve the order is obvious. As for the measure, for any $s<t\in \dot I$
$$
\mu^T\left(\dot\theta([s,t])\right) = dL\left(p_h^{-1}\left(\dot\theta([s,t])\right)\right) = dL\left(\left[J_s,J_t\right]\right)=t-s,
$$
which proves that the push forward of the Lebesgue measure by $\dot\theta$ is $\mu^T$. 
\end{proof}


\subsection{Coalescent point processes}

As we just saw in the proof of Theorem \ref{thm:comb-reducedtree}, the comb representation of the sphere of radius $T$ of a locally compact tree is given by the list of the depths of the excursions of $h$ away from $T$, where $h$ is the jumping contour process of the ball of radius $T$ (i.e., of the tree truncated at height $T$). More specifically, if $L$ is the local time at $T$ of $h$ and $J$ its right inverse
$$
J_t := \inf\{s>0:L_s >t\}, 
$$
then every jump time $s$ of $J$ corresponds to an interval $(g_s,d_s):=(J_{s-}, J_s)$ of constancy of $L$ such that $h\not=T$ on $(g_s,d_s)$ and $h(g_s-)=h(g_s)= T = h(d_s)$. The comb representation of the metric structure of the sphere of radius $T$ is given by the comb $f$ defined by 
$$
f(s):=
\left\{
\begin{array}{cl}
2\,(T-\inf_{[g_s, d_s]}h)&\text{ if }\Delta J_s \not=0\\
0 &\text{ otherwise.}
\end{array}
\right.
$$ 
\begin{dfn}
Let $\nu$ be a $\sigma$-finite measure on $(0,\infty)$ such that $\nu([\varepsilon,\infty))<\infty$ for all $\varepsilon>0$. Let $\mc$ be a Poisson point process on $(0,\infty)^2$ with intensity $\mbox{Leb}\,\otimes\,\nu$ and denote by $(S_i, H_i)_i$ its atoms. Finally, let $(D,H)$ denote the first atom (in the first dimension) such that $H>T$.
We will say that the random comb metric space associated with the comb $\sum_{i:S_i<D}2H_i\indicbis{S_i}$ is a \emph{coalescent point process} with height $T$ and intensity $\nu$.
\end{dfn}
\noindent
For a random $h$ which is a strong Markov process for which $T$ is a regular point, there is an alternative way to the one given in the appendix of constructing the local time $L$ at $T$, which ensures that it is adapted and so unique up to a multiplicative constant. Then if $J$ denotes the inverse of $L$, to each jump $\Delta J_s$ of $J$ corresponds an excursion of $h$ away from $T$ with length $\Delta J_s$, say $e_s$. Furthermore, $((s, e_s);s\ge 0)$ are the atoms of a Poisson point process in $(0,\infty)\times \Ec$, where $\Ec$ is the space of càdlàg paths with finite lifetime $V$ visiting $T$ at most at 0 and at $V$. The intensity measure, say $\mu$, of this Poisson point process is called Itô's excursion measure and is the analogue to the common probability distribution of excursions when the visit times of T by $X$ form a discrete set. 

Recall the comb $f$ before the definition. A consequence of what precedes is that when $h$ is a strong Markov process $((s,f(s));f(s)\not=0)$ are the atoms of a Poisson point process in $(0,\infty)^2$ whose intensity is $\mbox{Leb}\,\otimes\,\nu$, where $\nu$ is the push forward of $\mu$ by the map $e\mapsto 2(T-\inf_{[0, V]} e)$. 
A particular example of such a coalescent point process is given in the case of the Continuum Random Tree \citep{Ald91}, which we (abusively) define here as the tree coded by the Brownian excursion conditioned to hit $T$.
\begin{thm}[\citealt{Pop04, AP05}]
The sphere $\trT$ of the Brownian tree $\tr$ is a coalescent point process with height $T$ and intensity measure $\nu$, where 
\begin{equation}
\label{eqn:formula_0}
\nu(dh) = \frac{dh}{2 h^2} 
\end{equation}
\end{thm}
\noindent
Another nice example is the case of splitting trees, which form a class of non-Markovian branching processes in continuous time. Specifically, all individuals are assumed to have i.i.d. lifetimes during which they give birth at constant rate to independent copies of themselves. \cite{Lam10} proved that the jumping contour process of the truncation below $T$ of a splitting tree is a compound Poisson process with drift $-1$ reflected below $T$ and killed upon hitting 0. Then it is easy to see that the sphere of radius $T$ has a simple comb representation, in the form of i.i.d. excursion depths.\\

\cite{LUB15} have extended this study to all trees satisfying the following property. Conditional on any geodesic starting at the root, the set $(x_i, \tr_i)$ of branching points on this geodesic and of the subtrees grafted at each of these points, form a Poisson point process with intensity $\mbox{Leb}\,\otimes\,\xi$, where $\xi$ is a $\sigma$-finite measure on the space of locally compact trees. Specifically, we have characterized such trees as those trees whose truncations have as jumping contour process a reflected Lévy process with no negative jump. Therefore, the metric structure of the spheres of such trees has again a representation in the form of a coalescent point process.\\

Last, \citet{LP13} have expressed the distribution of the coalescent point process for non-binary branching trees, including Galton--Watson processes and continuous-state branching processes, which have branching points of arbitrarily large degree. In this case, the comb is not a Poisson point process because of the multiple points, but it has a nice regenerative property.

\paragraph{Acknowledgements.} AL  thanks the {\em Center for Interdisciplinary Research in Biology} (Coll{\`e}ge de France) for funding. The research of GUB is supported by UNAM-DGAPA-PAPIIT grant no. IA101014.

\bibliographystyle{apalike}
\bibliography{FromZotero} 

\appendix

\section{Finite ultrametric sets}
\label{sec:finite-ultra}

Here we prove the statement made in the introduction that for any finite subset $S$ of an ultrametric space $(U,d)$, there is at least one labelling of its elements $\{x_i:1\le i\le n\}$ such that  for any $i<j$, 
\begin{equation}
\label{eqn:comb-finite}
d(x_i,x_{j}) = \max\{d(x_k,x_{k+1});i\le k<j\}. 
\end{equation}
Let us proceed by induction on $n$. For $n=2$, the statement is trivial, but it will be useful later to have initialized the induction at $n=3$. 

Let $x,y,z\in U$. Assume that there are two pairs with different distances (otherwise the statement is obvious), for example $d(y,z)<d(x,z)$. We are going to show that this implies the equality $d(x,z) = d(x,y)$. First, $d(x,y)\le d(x,z)\vee d(z,y) = d(x,z)$, so that
$$
d(x,y)\le d(x,z).
$$
Second, $d(y,z)<d(x,z)\le d(x,y)\vee d(y,z)$, which yields $d(y,z)<d(x,z)\le d(x,y)$, so that
$$
d(y,z)<d(x,z)= d(x,y).
$$ 
So we can define for example $x_1=z$, $x_2=y$ and $x_3=x$, for then $d(x_1, x_3) = \max(d(x_1, x_2), d(x_2, x_3))$.

Now let $n\ge 3$ and assume that the induction property holds for subsets of $U$ with cardinality $n$. Let $S$ be a finite subset of $U$ with cardinality $n+1$. Let $y,z\in S$ be the pair realizing the minimum distance, i.e. such that $d(y,z)\le d(v,w)$ for any $v,w\in S$. Apply the induction property to $S\setminus \{z\}$, which has cardinality $n$. Then there is a labelling $\{x_i':1\le i\le n\}$ of $S\setminus \{z\}$ such that  Eq \eqref{eqn:comb-finite} holds (replacing all $x_j$ by $x_j'$). Let $K$ be the label of $y$ (i.e., $y = x_K'$) and set
$$
x_j:=\left\{\begin{array}{ccl}
x_j'&\mbox{if}& j\le K\\
z&\mbox{if}& j= K+1\\
x_{j-1}'& &\mbox{else.}
\end{array}
\right.
$$
Now since $y,z$ realize the minimum distance, Eq \eqref{eqn:comb-finite} holds for any $i<j$ such that $i$ and $j$ are both different from $K+1$. It holds also when $i$ or $j$ is equal to $K+1$ by an application of the argument initializing the induction, which ensures that for any $\ell\not\in\{K,K+1\}$, $d(x_\ell, z) =d(x_\ell,y)$.

\section{Local time} 
\label{sec:local time}
If $A$ is a closed subset of $[0,\infty)$, a \emph{local time} associated to $A$ is a nondecreasing mapping $L:[0,\infty)\to[0,\infty)$ such that $L(0)=0$ and whose set of points of increase coincides with $A$. If $A$ is discrete, $L$ can be defined simply, for example as the counting process $L_t=\#[0,t]\cap A$. But then if $A$ is not discrete, the counting process will blow up at the first accumulation point of $A$, so a different strategy is needed.\\ 

Assume that $A$ is perfect, i.e., it has empty interior and no isolated point. For any compact interval, say $[0,M]$, we can construct a continuous mapping $L:[0,M]\to [0,1]$ such that $L(0)=0$, $L(M)=1$ and for any $0\le s<t\le M$
$$
L_t>L_s \Leftrightarrow (s,t)\cap A\not=\varnothing.
$$
The construction is recursive, exactly as for the Cantor--Lebesgue function, also known as `devil's staircase'. 

First recall that the open set $B:=(0,M)\setminus A$ can be written as a countable union of open intervals, say $(I_n)_{n\ge 1}$. For each $n\ge 1$, write $I_n=(g_n,d_n)$. Note that for any $\varepsilon>0$, there can be only finitely many of these intervals which have length larger than $\varepsilon$. Therefore, we can assume that the intervals $(I_n)$ are ranked by decreasing order of their lengths (since ties can only occur in finite numbers, we can rank them by increasing order of their left-hand extremities $g_n$'s, say). 

We are going to construct recursively, for each $n\ge 1$, a continuous mapping $L^n:[0,M]\to [0,1]$ which is piecewise affine and constant exactly on $\cup_{k=1}^n I_k$. First, $L^1$ is the function equal to $1/2$ on $I_1$, affine on $[0,g_1]$ and on $[d_1,1]$, such that $L^1(0)=0$ and $L^1(M)=1$. Now assume that we are given a continuous function $L^n:[0,M]\to [0,1]$ which is piecewise affine and constant exactly on $\cup_{k=1}^n I_k$. Writing $d_0=0$ and $g_0=1$, there is a unique pair $0\le k,j\le n$ such that $d_j<g_{n+1}<d_{n+1}<g_k$ minimizing $g_k-d_j$. Then we can define $L^{n+1}$ as the continuous function equal to $L^n$ outside $(d_k,g_j)$, constant to $\frac12(L^n(g_j) + L^n(d_k))$ on $[g_{n+1}, d_{n+1}]$, affine on $[d_j, g_{n+1}]$ and on $[d_{n+1}, g_k]$. It is easy to see that $L^{n+1}$ satisfies the announced properties. 

Now for any $p\in\N$, let $\Dc_p$ denote the set of dyadic numbers of $(0,1)$ whose dyadic expansion has length smaller than $p$, i.e., $\Dc_p=\{x\in (0,1): \exists (x_1,\ldots,x_p) \in\{0,1\}:x=\sum_{k=1}^px_k\,2^{-k}\}$. Also for $n\in \N$, let $\Jc_n$ denote the set of values taken by $L^n$ on its constancy intervals. Because $A$ has no isolated point, for each $p\ge 1$ there is an integer $N$ such that for all $n\ge N$, $\Dc_p\subset\Jc_n$, so that for any $n'\ge N$, $\| L^n-L^{n'}\|\le 2^{-p}$, where $\|\cdot\|$ is the supremum norm on $[0,1]$. This shows that $(L^n)$ is a Cauchy sequence for the supremum norm, and so converges uniformly on $[0,1]$ to a continuous function $L$. It is not difficult to see, using the stationarity of $(L^n)$ on $B$, that $L$ increases exactly on $A$.

\end{document}

%% file: maquereaux.tex
\usepackage[ansinew]{inputenc}
\usepackage{amsfonts}
\usepackage{amssymb,amsmath}
\usepackage{bbm}
\usepackage[french, english]{babel}
\usepackage{latexsym}
\usepackage{natbib}

\usepackage{hyperref}
\usepackage[numbered]{bookmark}
\usepackage{stmaryrd}

\usepackage[T1]{fontenc}

\usepackage{amsfonts,amssymb, amsthm,amsmath, amscd, dsfont}
\usepackage{mathrsfs}
\usepackage{bbm}
\usepackage{graphicx}

\newcommand {\debeq}	{\begin{eqnarray*}}
\newcommand {\fineq}	{\end{eqnarray*}}

\newcommand	{\dtf}	{\bar{d}_f}
\newcommand	{\dtfp}	{\bar{d}_{F_p}}

\newcommand{\tr}{\mathbbm{t}}
\newcommand{\trT}{\tr^{\{T\}}}
\newcommand{\dottrT}{\dot\tr^{\{T\}}}

\newcommand{\mc}{\mathcal{M}}

\newcommand{\R}{\mathbb{R}}

\newcommand{\Z}{\mathbb{Z}}
\newcommand{\Q}{\mathbb{Q}}
\newcommand{\N}{\mathbb{N}}

\newcommand{\Ac}{\mathscr{A}}

\newcommand{\Dc}{\mathscr{D}}
\newcommand{\Ec}{\mathscr{E}}

\newcommand{\Jc}{\mathscr{J}}

\def \Ã©{\'e}
\def \Ã¨{\`{e}}
\def \Ã¹{\`{u}}
\def \Ã®{\^{\i}}
\def \Ã {\`{a}}
\def \Ã§{\c{c}}

\newtheorem	{thm}		{Theorem}[section]
\newtheorem	{dfn}		[thm]{Definition}
\newtheorem	{lem} 	[thm]	{Lemma}
\newtheorem	{prop}	[thm]{Proposition}

\newtheorem     {rem}       [thm]    {Remark}

\newcommand	{\indic}	[1] {{\mathbbm{1}}_{#1}}
\newcommand	{\indicbis}	[1] {\indic{\{#1\}}}

\providecommand{\Fonction}[5]
{\begin{alignat*}{2}
  & #1 :\; & #2 & \longrightarrow #3 \\
  & & #4 & \longmapsto #5
\end{alignat*}}